%% file: main.tex
\documentclass[a4paper,reqno,11pt]{amsart} 
\usepackage[T1]{fontenc}
\usepackage[utf8]{inputenc}

\usepackage{amsmath, amssymb, graphicx, enumerate, enumitem, color}
\usepackage[usenames,dvipsnames,svgnames,table]{xcolor}
\usepackage{svg}
\usepackage[foot]{amsaddr}
\usepackage[textsize=tiny]{todonotes}

\usepackage{multirow}
\usepackage{longtable}
\usepackage{booktabs}
\usepackage{comment}
\usepackage{attachfile2}

\usepackage[normalem]{ulem}

\usepackage{listings}
\usepackage[T1]{fontenc}
\usepackage{xcolor}
\usepackage{textcomp}

\definecolor{commentscolor}{rgb}{0,0.6,0}
\definecolor{keywordscolor}{HTML}{0512FF}
\definecolor{stringscolor}{HTML}{AB354D}
\definecolor{backcolor}{HTML}{FFFFFF}
\definecolor{framecolor}{HTML}{CCCCCC}

\lstdefinestyle{mystyle}{
    frame=single,
    framesep=5pt,
    rulecolor=\color{backcolor},
    upquote=true,
    backgroundcolor=\color{backcolor},   
    commentstyle=\color{commentscolor},
    keywordstyle=\color{keywordscolor},
    stringstyle=\color{stringscolor},
    basicstyle=\ttfamily,
    breakatwhitespace=false,         
    breaklines=true,                 
    captionpos=b,                    
    keepspaces=true,                                    
    showspaces=false,                
    showstringspaces=false,
    showtabs=false,                  
    tabsize=4
}
 
\makeatletter
\def\lst@outputspace{{\ifx\lst@bkgcolor\empty\color{white}\else\lst@bkgcolor\fi\lst@visiblespace}}
\makeatother
\definecolor{1}{HTML}{FF0303}
\definecolor{2}{HTML}{2323FF}
\definecolor{3}{HTML}{CB9154}
\definecolor{4}{HTML}{6AA84F}
\definecolor{5}{HTML}{BA54FF}
\definecolor{6}{HTML}{A9537E}

\newcommand{\1}{\color{1}1\color{black}}
\newcommand{\2}{\color{2}2\color{black}}
\newcommand{\3}{\color{3}3\color{black}}
\newcommand{\4}{\color{4}4\color{black}}
\newcommand{\5}{\color{5}5\color{black}}
\newcommand{\6}{\color{6}6\color{black}}
\newcommand{\s}{\hphantom{1}}
\makeatletter
\def\thm@space@setup{
  \thm@preskip=4mm
  \thm@postskip=0mm
}
\makeatother

\usepackage{tcolorbox}
\usepackage{hyperref}
\usepackage{xcolor}
\hypersetup{
  pdfauthor={Marcin Briański, Jędrzej Hodor, Hoang La, Piotr Micek, Katzper Michno},
  pdftitle={Boolean dimension of a Boolean lattice},
    colorlinks,
    linkcolor={RoyalBlue},
    citecolor={RubineRed},
    urlcolor={blue!80!black}
}
\usepackage[capitalise, noabbrev]{cleveref}
\usepackage{mathtools}
\DeclarePairedDelimiter\set{\{}{\}}

\theoremstyle{plain} 
\newtheorem{theorem}{Theorem}
\newtheorem{myquestion}{Question}
\newtheorem{corollary}[theorem]{Corollary}

\newtheorem{lemma}[theorem]{Lemma}

\newtheorem{proposition}[theorem]{Proposition}

\theoremstyle{remark} 
 
\bibstyle{plain}

\newcommand{\lat}[1]{\mathcal{B}_{#1}}
\newcommand{\mlat}[1]{\mathcal{M}_{#1}}
\newcommand{\mn}[2]{\mathcal{M}_{#1, #2}}

\newcommand{\cgS}{\mathcal{S}}

\newcommand{\bdim}{\operatorname{bdim}}

\graphicspath{{Figures/}}

\let\leq\leqslant
\let\geq\geqslant

\let\subset\subseteq

\let\epsilon\varepsilon

\renewenvironment{enumerate}{\begin{enumorig}[label=\textup{(\roman*)}, noitemsep, 
topsep=2pt plus 2pt, labelindent=.2em, leftmargin=*, widest=iii]}{\end{enumorig}}

\makeatletter 
\let\old@setaddresses\@setaddresses 
\def\@setaddresses{\bigskip\bgroup\parindent 0pt\let\scshape\relax\old@setaddresses\egroup}
\makeatother

\begin{document} 
\title[Boolean dimension of a Boolean lattice] 
{Boolean dimension of a Boolean lattice}

\author[M.~Briański]{Marcin Briański}
\email{marcin.brianski@doctoral.uj.edu.pl}
\author[J.~Hodor]{J\c{e}drzej Hodor}
\email{jedrzej.hodor@gmail.com}
\author[H.~La]{Hoang La}
\email{hoang.la.research@gmail.com}
\author[P.~Micek]{Piotr Micek}
\email{piotr.micek@uj.edu.pl}
\author[K.~Michno]{Katzper Michno}
\email{katzper.michno@gmail.com}

\address[M.~Briański, J.~Hodor, H.~La, P. ~Micek, K. ~Michno]{Theoretical Computer Science Department\\ 
  Jagiellonian University\\
  Kraków, Poland}
  
\thanks{M.\ Briański, J.\ Hodor, H.\ La, and P.\ Micek are partially supported by a Polish National Science Center grant (BEETHOVEN; UMO-2018/31/G/ST1/03718).}

\subjclass[2010]{06A07} 
  
\keywords{Poset, dimension, Boolean dimension, Boolean lattice}

\begin{abstract}
    For every integer $n$ with $n \geq 6$, we prove that the Boolean dimension of a poset consisting of all the subsets of $\{1,\dots,n\}$ equipped with the inclusion relation is strictly less than $n$.
\end{abstract}

\maketitle

\section{Introduction}\label{sec:introduction}
\input{./s.introduction.tex}

\section{Preliminaries}\label{sec:definitions}
\input{./s.definitions.tex}

\section{Boolean dimension and dimension of products of posets}\label{sec:product}
\input{./s.product.tex}
\section{Boolean dimension of the Boolean lattice}\label{sec:upperbound}
\input{./s.upperbound.tex}

\section{Posets of multisets}\label{sec:multisets}
\input{./s.multisets.tex}

\section{Open problems}\label{sec:open_problems}
\input{./s.open.tex}

\section*{Acknowledgments}
We thank the anonymous reviewers and George Bergman for their valuable feedback, which helped improve this paper.

\bibliographystyle{abbrv}
\bibliography{booleandimension}

\newpage

\appendix

\section{The code to verify \cref{lem:bdim_of_B_6}}\label{sec:code}
\input{./s.code.tex}

\end{document}

%% file: s.introduction.tex
The most widely studied measure of complexity of partially ordered sets (posets for short) is their dimension, 
introduced by Dushnik and Miller~\cite{DM41} in 1941.
Low-dimensional posets admit a concise procedure for handling comparability queries of the form "is $x\leq y$?".
In the 1980s, Gambosi, Nešetřil, Pudl\'{a}k, and Talamo \cite{GNT90, NP89} introduced the notion of Boolean dimension, which generalizes the notion of dimension with emphasis on the existence of the mentioned compact schemes.
See \cref{sec:definitions} for definitions of dimension and Boolean dimension.
For a poset $P$, we write $\dim(P)$ for the dimension of $P$ and $\bdim(P)$ for the Boolean dimension of~$P$.
The most prominent (and beautiful) open problem on Boolean dimension comes from the initial paper by Nešetřil and Pudl\'{a}k~\cite{NP89}: Do posets with planar cover graphs have bounded Boolean dimension?
For recent progress towards resolving this problem see \cite{FMM20, BMT22}.


In this paper, we consider the following question: 
What is the Boolean dimension of a Boolean lattice?
For a positive integer $n$, the \emph{Boolean lattice of order $n$}, denoted by $\lat{n}$, is the poset on all the subsets of $[n] = \{1,\dots,n\}$ ordered by the inclusion relation.
Although it is well-known and easy to see that $\dim(\lat{n}) = n$,  
the problem of determining the dimension of the union of two levels of $\lat{n}$ had been heavily studied in the 1990s -- see e.g.\ \cite{BKKT94, F94, K96, K97, K99, HKT00}.
Recently, this area of study was revisited due to an increasing interest in the divisibility orders -- see~\cite{H22, LS21, SV22}.


In the founding paper of the poset dimension theory, Dushnik and Miller introduced the family of posets $\{S_n : n \geq 2\}$, later referred to as standard examples \cite[Theorem~4.1]{DM41}.
The poset $S_n$ is isomorphic to the subposet of $\lat{n}$ induced by all singletons and all co-singletons -- see~\cref{fig:1}. 
The striking feature of the family of standard examples is that $\dim(S_n) = n$. 
On the other hand, the Boolean dimension of every standard example is at most $4$.
Since $\bdim(P)\leq\dim(P)$ for all posets $P$, 
we have $\bdim(\lat{n})\leq n$ for all positive integers $n$.
The family of Boolean lattices had appeared as a natural candidate to be the canonical example of a family with the property that $\bdim(\lat{n}) = n$.
The question, of whether this is true was circulating in the community since Order \& Geometry Workshop 2016 held in Poland. In writing, it appeared e.g.\ in \cite[Section~3.4]{BPSTT19} and \cite[page~6]{KT23}.
We answer the question in the negative, that is, we prove the following.

\begin{theorem}\label{th:upper_bound}
    For every integer $n$ with $n \geq 6$,
    $$\bdim(\lat{n}) < n.$$
\end{theorem}

More precisely, we prove that $\bdim(\lat{n}) \leq \left\lceil\frac{5}{6} n\right\rceil$ for every positive integer $n$. 
The best lower bound we could prove is $\bdim(\lat{n}) \geq n \slash \log(n+1) $ -- 
see \cref{cor:lat_bdim_lower_bound}. Actually, we tend to believe that the right order of magnitude is $o(n)$.

\begin{figure}[t!]
    \centering
    \includegraphics[scale=0.9]{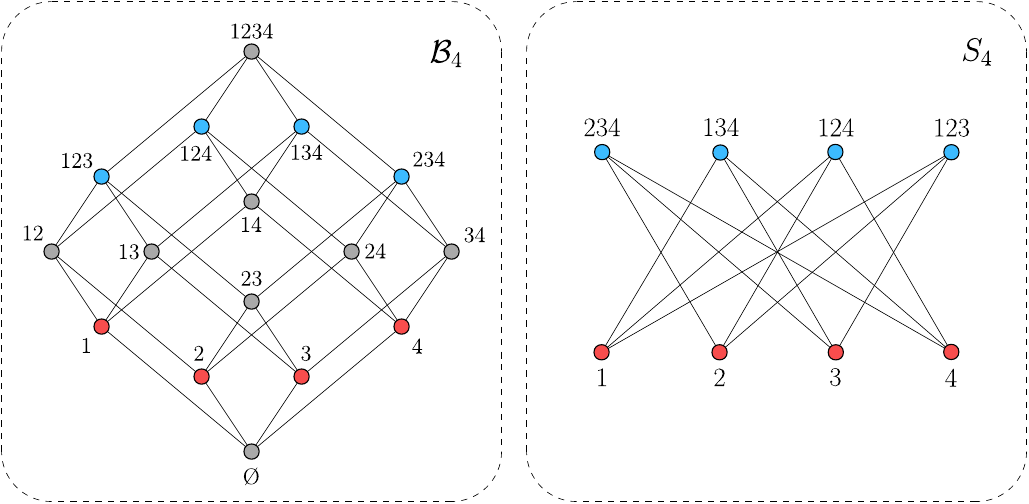}
    \caption{
    For every positive integer $n \geq 2$, the standard example of order $n$, denoted by $S_n$, is isomorphic to the subposet of $\lat{n}$ induced by the singletons and co-singletons.
    On the right, we show a poset diagram of $S_4$ and on the left, we show $S_4$ as a subposet of $\lat{4}$.
    }\label{fig:1}
\end{figure}

Next, we give an example of a family of posets with the above-mentioned property for the Boolean dimension.
This family is a natural generalization of the family of Boolean lattices.
For every positive integer $n$, we define the poset $\mlat{n}$, as the poset on the family of all multisets containing elements in $[n]$ equipped with the inclusion relation. In a multiset, we allow elements to have arbitrary multiplicities; thus, these posets have infinitely many elements. 
Interestingly, the posets $\mlat{n}$ are particularly useful in studying the dimension of divisibility posets.
It is not hard to see that $\mlat{n}$ is a product of $n$ infinite linear orders, and so, $\bdim(\mlat{n}) \leq \dim(\mlat{n}) \leq n$ (see \cref{sec:product} for more details).
We prove that this bound is tight.

\begin{theorem}\label{th:mlat_exact_bdim}
    For every positive integer $n$,
        \[\bdim(\mlat{n}) = \dim(\mlat{n}) = n.\]
\end{theorem}

This result can be derived using the Product Ramsey Theorem as mentioned in \cite[Section~3.4]{BPSTT19} (for more on this version of Ramsey Theorem see \cite[Section~3.1]{BPSTT19}, \cite[Section~4]{FFT99}, and \cite{TWW23}).
We provide a simple elementary proof, which also implies the mentioned lower bound: $\bdim(\lat{n}) \geq n\slash \log (n+1)$.
To be more precise, we prove that for every positive integer $n$, there exists a positive integer $m$ such that the Boolean dimension of a subposet of $\mlat{n}$, induced by all multisets with all elements having multiplicities less than $m$, is at least $n$.
We give an explicit upper bound on $m$, namely, $m\leq n^{n-1}$.

The rest of this paper is organized as follows. In \cref{sec:definitions}, we provide some essential definitions and notations used throughout the rest of the paper. In \cref{sec:product}, we discuss the dimension and the Boolean dimension of products of posets. In \cref{sec:upperbound}, we prove \cref{th:upper_bound}. In \cref{sec:multisets}, we prove \cref{th:mlat_exact_bdim}. Finally, in \cref{sec:open_problems}, we present some related open problems.

%% file: s.definitions.tex
The set of the first $n$ positive integers is denoted by $[n]$. By $\log x$ we denote the logarithm of $x$ with base $2$. 

A \emph{partially ordered set}, or \emph{poset} for short, is an ordered pair $P = (X, \leq)$, where $X$ is a non-empty set of elements called the \emph{ground set of $P$}, and $\leq$ is a binary relation on $X$ (called the \emph{order relation in $P$}), which is reflexive, antisymmetric and transitive. 
We do not require ground sets to be finite.
Sometimes, we replace the phrase $x \leq y$ in $P$ with $x \leq_P y$. 
For two posets $P$ and $Q$, we say that $Q$ is a \emph{subposet} of $P$ (denoted by $Q \subseteq P$) if 
the ground set of $Q$ is a subset of the ground set of $P$ and the order relation of $Q$ is the restriction of the order relation in $P$ to the ground set of $Q$.
We say that two elements $x,y$ in a poset $P$ are \emph{comparable} if $x \leq_P y$ or $y \leq_P x$. A poset, where all pairs of elements are comparable is called a \emph{linear order}. 
A poset $\overline{P}$ is a \emph{linear extension} of $P$ if $\overline{P}$ is a linear order on the ground set of $P$ such that $x \leq y$ in $\overline{P}$ whenever $x \leq y$ in $P$ for every two elements $x, y$ in $P$.

Let $P$ and $Q$ be two posets with ground sets $X$ and $Y$, respectively.
The \emph{product of $P$ and $Q$}, denoted by $P \times Q$, is the poset with the ground set $X \times Y$, where for any two pairs $(x_1, y_1), (x_2, y_2) \in X \times Y$, we have $(x_1, y_1) \leq (x_2, y_2)$ in $P \times Q$ if and only if $x_1 \leq x_2$ in $P$ and $y_1 \leq y_2$ in $Q$. 
Let $n$ be a positive integer. 
The \emph{$n$-th power of $P$} is the product of $n$ copies of $P$, denoted by $P^n$. 
The posets $P$ and $Q$ are \emph{isomorphic} if there exists a bijection $g: X \rightarrow Y$ such that for every two elements $x, y$ in $P$, we have $x \leq y$ in $P$ if and only if $g(x) \leq g(y)$ in $Q$.

For a linear order $L$ and two elements $x,y$ in $L$, we define $[ x \leq_{L} y ] \in \{ 0, 1 \}$ to be $1$ if $x \leq y$ in $L$ and $0$ otherwise.
For a sequence of linear orders $L_1, \dots,L_n$ on the same ground set and two elements $x,y$ in the ground set, we abbreviate
    \[[ x  \leq_{L_i} y ]_{i = 1}^n = ([ x  \leq_{L_1} y ], \dots, [ x  \leq_{L_n} y ]) \in \{0,1\}^{n}.\]
Let $P$ be a poset with at least two elements, and let $d$ be a nonnegative integer. 
Let $L_1, \dots, L_d$ be linear orders on the ground set of $P$, and let $\phi: \{0, 1\}^{d} \rightarrow \{ 0, 1 \}$ be any map. 
The pair $((L_1, \dots, L_d), \phi)$ is a \emph{Boolean realizer of $P$} if for every pair of elements $x, y$ in $P$,
    \[ x \leq_P y \hspace{3mm} \mathrm{iff} \hspace{3mm} \phi \left([x \leq_{L_i} y]_{i = 1}^d \right) = 1. \]
The \emph{size} of a Boolean realizer is the number of linear orders in the realizer. 
The \emph{Boolean dimension} of $P$ is equal to the minimum size of a Boolean realizer of $P$.\footnote{Note that in the literature two (almost equivalent) variants of this definition appear. See~\cite[Section~6]{Bergman25} for more details.}
The \emph{dimension of $P$} can be defined as the minimum size of a Boolean realizer of $P$, where the formula $\phi$ is fixed to be $\phi(\varepsilon_1, \dots, \varepsilon_d) = \varepsilon_1 \cdot \ldots \cdot \varepsilon_d$. 
Note that this is not the usual way of phrasing this definition (see e.g.\ \cite{Tro19} for the classical definition and basic facts on the dimension). 
However, this immediately yields $\bdim(P) \leq \dim(P)$ for every poset $P$. 
It is not hard to see that both Boolean dimension and dimension are monotone under taking subposets.

%% file: s.product.tex
One of the exercises one can solve to familiarize oneself with the notion of the dimension of a poset is to show that for every two posets $P$ and $Q$, we have $\dim(P \times Q) \leq  \dim(P) + \dim(Q)$.
In particular, this implies that for every poset $P$ and every positive integer~$n$, we have $\dim(P^n) \leq n \cdot \dim(P)$.
Therefore, to upper bound the dimension of a poset, one can express this poset in terms of products of some other posets.
It turns out that both $\lat{n}$ and $\mlat{n}$ have very natural representations as products of $n$ linear orders (two-element linear orders in the former case and infinite linear orders in the latter case). 
Clearly, $\lat{n}$ is a subposet of $\mlat{n}$.

Let $n$ be a positive integer.
Consider a bijection between all elements of $\mlat{n}$ and $\{0,1,\dots\}^{n}$ defined as follows.
To a given multiset $A$ of elements in $[n]$, assign $v \in \{0,1,\dots\}^{n}$, where for each $i \in [n]$, the value $v_i$ is the multiplicity of $i$ in $A$.
The bijection transforms the inclusion relation into the coordinate-wise order relation, in other words, the product relation. 
We obtain that $\mlat{n}$ is isomorphic to $(\mlat{1})^n$, which implies $\dim(\lat{n}) \leq \dim(\mlat{n}) \leq n$. 

On the other hand, as was already mentioned, $\lat{n}$ contains the standard example of order $n$ as a subposet, which yields $n \leq \dim(\lat{n})$.

\begin{proposition}\label{prop:dim_of_lattices}
    For every positive integer $n$,
        \[ \dim(\lat{n}) = \dim(\mlat{n}) = n. \]
\end{proposition}

We can also prove the additive property of the Boolean dimension for product of posets.
We will use this property in the proof of \cref{th:upper_bound}.

\begin{lemma}\label{lemma:bdim_prod}
    For every two posets $P$ and $Q$, we have $\bdim(P \times Q) \leq \bdim(P) + \bdim(Q)$.
\end{lemma}

\begin{proof}
    Let $P$ and $Q$ be two posets.
    If $P$ has exactly one element, then $P \times Q$ is isomorphic to $Q$, and so, $\bdim(P \times Q) = \bdim(Q) \leq \bdim(P) + \bdim(Q)$.
    Symmetrically, the assertion follows in the case where $Q$ has exactly one element.
    Now, we assume that both posets have at least two elements. 
    Let $s = \bdim(P)$ and $t = \bdim(Q)$.
    Let $((L_1, \dots, L_s), \phi_P)$ be a Boolean realizer of $P$ and $((K_1, \dots, K_t), \phi_Q)$ be a Boolean realizer of $Q$.

    The goal is to define a Boolean realizer of $P \times Q$ of size $s + t$.
    We start with $\phi: \{0, 1\}^{s+t} \rightarrow \{0, 1\}$ defined as
    \[
        \phi(\delta_1, \dots, \delta_s, \varepsilon_1, \dots, \varepsilon_t) = \phi_P(\delta_1, \dots, \delta_s) \cdot \phi_Q(\varepsilon_1, \dots, \varepsilon_t).
    \]
    Fix $\overline{P}$ and $\overline{Q}$ arbitrary linear extensions of $P$ and $Q$ respectively. 
    We define two families of linear orders on the ground set of $P \times Q$. 
    First, for each $i \in [s]$, we construct a linear order $M_i$.
    Let $(p_1, q_1), (p_2, q_2)$ be two elements of $P \times Q$.
    If $p_1 \neq p_2$, then we order the elements according to $L_i$, that is, $(p_1, q_1) \leq (p_2, q_2)$ in $M_i$ if and only if $p_1 \leq p_2$ in $L_i$.
    In the case where $p_1 = p_2$, we order the elements according to $\overline{Q}$.
    Next, for each $j \in [t]$, we construct a linear order $N_j$ similarly, that is, for all $(p_1, q_1), (p_2, q_2)$ elements of $P \times Q$, if $q_1 \neq q_2$, then we order the elements in $N_j$ as in $K_j$ and if $q_1 = q_2$, then we order the elements in $N_j$ as in $\overline{P}$.
    
    Let $(T_1,\dots,T_{s+t})$ be the concatenation of $(M_1,\dots,M_s)$ and $(N_1,\dots,N_t)$.
    One can verify that for all $(p_1,q_1)$,$(p_2,q_2)$ in $P\times Q$, 
    \begin{align*}
        \phi\left([(p_1,q_1) \leq_{T_i} (p_2,q_2)]_{i = 1}^{s+t} \right) = 
        \phi_P\left([p_1 \leq_{M_i} p_2]_{i = 1}^s \right) \cdot \phi_Q\left([q_1 \leq_{N_i} q_2]_{i = 1}^t \right).
    \end{align*}
    Therefore, $((T_1,\dots,T_{s+t}), \phi)$ is a Boolean realizer of $P \times Q$, and so, $\bdim(P \times Q) \leq s + t$.
\end{proof}

%% file: s.upperbound.tex
In this section, we prove the following result, which immediately implies \cref{th:upper_bound}.

\begin{theorem}\label{th:technical_upper_bound}
    For every positive integer $n$, $\bdim(\lat{n}) \leq \left\lceil\frac{5}{6} n\right\rceil$.
\end{theorem}

In fact, we just show that $\bdim(\lat{6}) \leq 5$, which combined with \cref{lemma:bdim_prod}, implies \cref{th:technical_upper_bound}.
Indeed, for every positive integer $n$, where $n = 6k + r$ for some nonnegative integer $k$ and $r \in \{0,\dots,5\}$, we have $\lat{n} = (\lat{6})^k \times \lat{r}$. Hence, 
    \[\bdim(\lat{n}) \leq k \cdot \bdim(\lat{6}) + \bdim(\lat{r}) \leq k \cdot 5 + r \leq \left\lceil\frac{5}{6} n\right\rceil. \]

\begin{lemma}\label{lem:bdim_of_B_6}
    It holds that $\bdim(\lat{6}) \leq 5$.
\end{lemma}

\begin{proof}
    Consider $\phi : \{0, 1 \}^{[5]} \rightarrow \{ 0, 1\}$, such that $\phi(\varepsilon_1, \varepsilon_2, \varepsilon_3, \varepsilon_4, \varepsilon_5) = 1$ if and only if there is at most one $0$ among $ \varepsilon_1, \varepsilon_2, \varepsilon_3, \varepsilon_4, \varepsilon_5 $. 
    In \cref{tab:boolean_realizer}, we give linear orders $L_1, L_2, L_3, L_4, L_5$ on the ground set of $\lat{6}$. 
    We claim that $((L_1, L_2, L_3, L_4, L_5), \phi)$ is a Boolean realizer of $\lat{6}$. 
    The claim can be verified using the Python script provided in \cref{sec:code}.
\end{proof}

The realizer in \cref{tab:boolean_realizer} was found using a SAT solver~\cite{SoosNC09}.
Our approach encodes the linear orders
and the formula directly.
For a given poset $P$ and a positive integer $k$, we build a SAT formula so that the formula is satisfiable if and only if $\bdim(P) \leq k$.
Moreover, given a satisfying assignment to the formula we can construct a Boolean realizer
\(((L_1, L_2, \dots, L_k), \phi)\) of $P$ of size $k$.
Introduce a variable $x_{A, B, i}$ for every pair of distinct elements $A, B $ in $P$ and
$i \in [k]$.
We will set $A <_{L_i} B$ in the Boolean realizer if and only if $x_{A, B, i}$ is set to $1$ in the satisfying assignment.
Next, introduce a variable $y_{m}$ for every $m \in \set{0,1}^k$.
We will set $\phi(m) = 1$ if and only if $y_m$ is set to $1$ in the satisfying assignment.
The SAT formula is a conjunction of the following conditions.
First, we have to make sure that each $L_i$ is a linear order, that is, for all distinct
$A,B,C$ in $P$ and $i \in [k]$ we add the clauses $x_{A,B,i} \wedge x_{B,C,i} \Rightarrow x_{A, C, i}$
and $x_{A,B,i} \Leftrightarrow \neg x_{B, A, i}$ to our formula.
Second, we have to make sure that our boolean formula \(\phi\) gives the correct answers for any input.
Thus, for every $m \in \set{0,1}^k$ and all $A, B$ in $P$ we add a condition saying that 
if $[A \leq_{L_i} B]_{i=1}^k = m$, then $y_m = 1$ if and only if $A \leq_P B$. 
For concreteness,
if $k=4$, $m=0110$ and $A \leq_P B$ the desired condition translates to
\[\neg x_{A, B, 1} \wedge x_{A, B, 2} \wedge x_{A, B, 3} \wedge \neg x_{A, B, 4} \Rightarrow y_{0110}.\]
Using this approach we were able to determine that for every $n \in [5]$ 
we have $\bdim(\lat{n}) = n$, $\bdim(\lat{6}) = 5$, and $\bdim(\lat{7}) = 6$,
making the inequality in \cref{th:technical_upper_bound} tight for $n \leq 7$.
We were unable to ascertain if $\bdim(\lat{8}) \geq 7$ in reasonable time. 
We remark that the SAT solver was usually able to quickly find a solution if it exists
(e.g. finding the realizer for $\lat{6}$ takes less than one minute on a modern desktop 
despite the formula having over $10000$ variables).
The code used to generate the SAT formulas is available in a public repository
\footnote{\href{https://gitlab.com/MrOverlord/bdim-finder}{https://gitlab.com/MrOverlord/bdim-finder}}.

\begin{table}[]
    \scriptsize
    \sffamily
    \centering
    \begin{tabular}{|c|c|c|c|c|}
        \hline
        \rule[-1.5ex]{1pt}{0pt} $L_1$ \rule{0pt}{3ex} & $L_2$ & $L_3$ & $L_4$ & $L_5$ \\
        {[}\1\s\2\s\3\s\4\s\5\s\6{]} & {[}\1\s\2\s\3\s\4\s\5\s\6{]} & {[}\1\s\2\s\3\s\4\s\5\s\6{]} & {[}\1\s\2\s\3\s\4\s\5\s\6{]} & {[}\1\s\2\s\3\s\4\s\5\s\6{]} \\
        {[}\1\s\2\s\3\s\s\s\5\s\6{]} & {[}\1\s\2\s\3\s\4\s\5\s\s{]} & {[}\s\s\2\s\3\s\4\s\5\s\6{]} & {[}\1\s\s\s\3\s\4\s\5\s\6{]} & {[}\1\s\s\s\3\s\4\s\5\s\6{]}\\
        {[}\1\s\s\s\s\s\4\s\5\s\s{]} & {[}\1\s\s\s\3\s\4\s\5\s\s{]} & {[}\s\s\s\s\3\s\4\s\5\s\6{]} & {[}\1\s\s\s\3\s\4\s\s\s\6{]} & {[}\1\s\2\s\3\s\s\s\5\s\6{]}\\
        {[}\s\s\s\s\s\s\4\s\5\s\s{]} & {[}\s\s\2\s\3\s\4\s\5\s\6{]} & {[}\s\s\s\s\3\s\4\s\s\s\6{]} & {[}\1\s\s\s\3\s\s\s\5\s\6{]} & {[}\s\s\s\s\s\s\4\s\5\s\6{]}\\
        {[}\1\s\2\s\3\s\s\s\5\s\s{]} & {[}\s\s\2\s\3\s\4\s\5\s\s{]} & {[}\s\s\s\s\s\s\4\s\s\s\s{]} & {[}\1\s\s\s\3\s\s\s\s\s\6{]} & {[}\s\s\s\s\s\s\4\s\s\s\6{]}\\
        {[}\1\s\s\s\3\s\s\s\5\s\s{]} & {[}\1\s\2\s\3\s\4\s\s\s\s{]} & {[}\1\s\2\s\3\s\s\s\5\s\6{]} & {[}\s\s\s\s\3\s\s\s\s\s\6{]} & {[}\1\s\2\s\3\s\4\s\5\s\s{]}\\
        {[}\s\s\s\s\3\s\s\s\5\s\s{]} & {[}\1\s\s\s\3\s\4\s\5\s\6{]} & {[}\s\s\2\s\3\s\s\s\5\s\6{]} & {[}\1\s\2\s\s\s\4\s\5\s\6{]} & {[}\1\s\s\s\3\s\4\s\5\s\s{]}\\
        {[}\1\s\2\s\s\s\s\s\5\s\6{]} & {[}\1\s\2\s\3\s\4\s\s\s\6{]} & {[}\s\s\s\s\3\s\s\s\5\s\6{]} & {[}\1\s\s\s\s\s\4\s\5\s\6{]} & {[}\1\s\s\s\3\s\s\s\5\s\6{]}\\
        {[}\1\s\s\s\s\s\s\s\5\s\s{]} & {[}\1\s\s\s\3\s\4\s\s\s\6{]} & {[}\1\s\2\s\3\s\s\s\s\s\6{]} & {[}\1\s\2\s\s\s\s\s\5\s\6{]} & {[}\1\s\s\s\3\s\s\s\5\s\s{]}\\
        {[}\s\s\2\s\3\s\s\s\5\s\6{]} & {[}\1\s\s\s\3\s\4\s\s\s\s{]} & {[}\s\s\2\s\3\s\s\s\s\s\6{]} & {[}\1\s\s\s\s\s\s\s\5\s\6{]} & {[}\s\s\s\s\3\s\s\s\5\s\6{]}\\
        {[}\s\s\2\s\s\s\s\s\5\s\6{]} & {[}\1\s\s\s\3\s\s\s\s\s\s{]} & {[}\1\s\2\s\s\s\4\s\5\s\6{]} & {[}\1\s\2\s\s\s\4\s\s\s\6{]} & {[}\1\s\s\s\s\s\4\s\5\s\6{]}\\
        {[}\s\s\2\s\3\s\s\s\5\s\s{]} & {[}\s\s\2\s\3\s\4\s\s\s\6{]} & {[}\1\s\2\s\s\s\s\s\5\s\6{]} & {[}\1\s\s\s\s\s\4\s\s\s\6{]} & {[}\1\s\s\s\s\s\s\s\5\s\6{]}\\
        {[}\1\s\2\s\s\s\4\s\5\s\s{]} & {[}\s\s\2\s\3\s\4\s\s\s\s{]} & {[}\s\s\2\s\s\s\4\s\5\s\6{]} & {[}\1\s\s\s\s\s\s\s\s\s\6{]} & {[}\s\s\s\s\s\s\s\s\5\s\6{]}\\
        {[}\s\s\2\s\s\s\4\s\5\s\s{]} & {[}\s\s\2\s\3\s\s\s\s\s\s{]} & {[}\s\s\2\s\s\s\s\s\5\s\6{]} & {[}\1\s\2\s\3\s\4\s\5\s\s{]} & {[}\1\s\2\s\3\s\4\s\s\s\6{]}\\
        {[}\1\s\2\s\s\s\s\s\5\s\s{]} & {[}\s\s\s\s\3\s\4\s\5\s\6{]} & {[}\1\s\2\s\3\s\4\s\s\s\6{]} & {[}\1\s\2\s\s\s\s\s\5\s\s{]} & {[}\1\s\2\s\3\s\s\s\s\s\6{]}\\
        {[}\s\s\s\s\3\s\4\s\5\s\s{]} & {[}\s\s\s\s\3\s\4\s\s\s\6{]} & {[}\s\s\s\s\s\s\s\s\5\s\6{]} & {[}\1\s\s\s\3\s\4\s\5\s\s{]} & {[}\1\s\2\s\s\s\s\s\s\s\s{]}\\
        {[}\s\s\2\s\s\s\s\s\5\s\s{]} & {[}\s\s\s\s\3\s\4\s\s\s\s{]} & {[}\1\s\2\s\s\s\4\s\s\s\6{]} & {[}\1\s\2\s\s\s\4\s\5\s\s{]} & {[}\1\s\s\s\3\s\4\s\s\s\6{]}\\
        {[}\s\s\s\s\s\s\s\s\5\s\s{]} & {[}\1\s\2\s\3\s\s\s\s\s\s{]} & {[}\s\s\2\s\3\s\4\s\s\s\6{]} & {[}\1\s\s\s\s\s\4\s\5\s\s{]} & {[}\1\s\s\s\s\s\4\s\s\s\6{]}\\
        {[}\1\s\2\s\3\s\4\s\s\s\6{]} & {[}\s\s\s\s\3\s\s\s\s\s\s{]} & {[}\s\s\s\s\3\s\s\s\s\s\6{]} & {[}\1\s\s\s\s\s\s\s\5\s\s{]} & {[}\1\s\s\s\3\s\s\s\s\s\6{]}\\
        {[}\1\s\2\s\s\s\4\s\s\s\6{]} & {[}\1\s\2\s\s\s\4\s\5\s\6{]} & {[}\s\s\2\s\s\s\4\s\s\s\6{]} & {[}\1\s\2\s\s\s\4\s\s\s\s{]} & {[}\1\s\2\s\s\s\s\s\s\s\6{]}\\
        {[}\1\s\2\s\3\s\4\s\s\s\s{]} & {[}\s\s\2\s\s\s\4\s\5\s\6{]} & {[}\1\s\2\s\s\s\s\s\s\s\6{]} & {[}\1\s\s\s\s\s\s\s\s\s\s{]} & {[}\1\s\s\s\s\s\s\s\s\s\6{]}\\
        {[}\1\s\2\s\s\s\4\s\s\s\s{]} & {[}\1\s\s\s\s\s\4\s\5\s\6{]} & {[}\s\s\2\s\s\s\s\s\s\s\6{]} & {[}\s\s\2\s\3\s\4\s\5\s\6{]} & {[}\s\s\s\s\s\s\s\s\s\s\6{]}\\
        {[}\1\s\2\s\3\s\s\s\s\s\6{]} & {[}\s\s\s\s\s\s\4\s\5\s\6{]} & {[}\s\s\2\s\3\s\4\s\5\s\s{]} & {[}\1\s\s\s\s\s\4\s\s\s\s{]} & {[}\1\s\2\s\3\s\4\s\s\s\s{]}\\
        {[}\1\s\2\s\3\s\s\s\s\s\s{]} & {[}\1\s\2\s\s\s\4\s\s\s\6{]} & {[}\1\s\s\s\3\s\s\s\5\s\6{]} & {[}\s\s\2\s\s\s\4\s\5\s\6{]} & {[}\1\s\s\s\3\s\4\s\s\s\s{]}\\
        {[}\1\s\2\s\s\s\4\s\5\s\6{]} & {[}\1\s\s\s\s\s\4\s\s\s\6{]} & {[}\1\s\s\s\3\s\s\s\s\s\6{]} & {[}\s\s\2\s\3\s\4\s\5\s\s{]} & {[}\1\s\s\s\s\s\4\s\s\s\s{]}\\
        {[}\1\s\2\s\s\s\s\s\s\s\6{]} & {[}\s\s\2\s\s\s\4\s\s\s\6{]} & {[}\s\s\s\s\s\s\s\s\s\s\6{]} & {[}\s\s\2\s\s\s\4\s\5\s\s{]} & {[}\1\s\2\s\3\s\s\s\5\s\s{]}\\
        {[}\1\s\2\s\s\s\s\s\s\s\s{]} & {[}\s\s\s\s\s\s\4\s\s\s\6{]} & {[}\1\s\2\s\3\s\4\s\5\s\s{]} & {[}\s\s\2\s\s\s\4\s\s\s\6{]} & {[}\1\s\2\s\3\s\s\s\s\s\s{]}\\
        {[}\s\s\2\s\3\s\4\s\s\s\6{]} & {[}\s\s\s\s\3\s\4\s\5\s\s{]} & {[}\1\s\2\s\3\s\4\s\s\s\s{]} & {[}\s\s\2\s\s\s\4\s\s\s\s{]} & {[}\1\s\s\s\3\s\s\s\s\s\s{]}\\
        {[}\s\s\2\s\3\s\s\s\s\s\6{]} & {[}\1\s\2\s\s\s\4\s\5\s\s{]} & {[}\s\s\2\s\3\s\4\s\s\s\s{]} & {[}\s\s\2\s\s\s\s\s\5\s\6{]} & {[}\s\s\2\s\3\s\4\s\5\s\s{]}\\
        {[}\s\s\2\s\s\s\4\s\5\s\6{]} & {[}\1\s\2\s\s\s\4\s\s\s\s{]} & {[}\1\s\2\s\3\s\s\s\5\s\s{]} & {[}\s\s\2\s\s\s\s\s\5\s\s{]} & {[}\1\s\s\s\s\s\s\s\s\s\s{]}\\
        {[}\s\s\2\s\3\s\4\s\s\s\s{]} & {[}\s\s\2\s\s\s\4\s\5\s\s{]} & {[}\1\s\2\s\s\s\4\s\5\s\s{]} & {[}\s\s\2\s\s\s\s\s\s\s\6{]} & {[}\s\s\2\s\3\s\4\s\5\s\6{]}\\
        {[}\s\s\2\s\3\s\s\s\s\s\s{]} & {[}\s\s\2\s\s\s\4\s\s\s\s{]} & {[}\1\s\2\s\s\s\4\s\s\s\s{]} & {[}\s\s\2\s\s\s\s\s\s\s\s{]} & {[}\s\s\2\s\3\s\4\s\s\s\6{]}\\
        {[}\s\s\2\s\s\s\4\s\s\s\6{]} & {[}\1\s\s\s\s\s\4\s\5\s\s{]} & {[}\s\s\2\s\s\s\4\s\s\s\s{]} & {[}\s\s\s\s\3\s\4\s\5\s\6{]} & {[}\s\s\2\s\3\s\4\s\s\s\s{]}\\
        {[}\s\s\2\s\s\s\4\s\s\s\s{]} & {[}\s\s\s\s\s\s\4\s\5\s\s{]} & {[}\1\s\2\s\s\s\s\s\5\s\s{]} & {[}\s\s\s\s\3\s\4\s\5\s\s{]} & {[}\s\s\2\s\3\s\s\s\5\s\6{]}\\
        {[}\s\s\2\s\s\s\s\s\s\s\6{]} & {[}\1\s\s\s\s\s\4\s\s\s\s{]} & {[}\1\s\2\s\3\s\s\s\s\s\s{]} & {[}\s\s\s\s\3\s\4\s\s\s\s{]} & {[}\s\s\2\s\3\s\s\s\s\s\6{]}\\
        {[}\s\s\2\s\s\s\s\s\s\s\s{]} & {[}\1\s\2\s\3\s\s\s\5\s\6{]} & {[}\1\s\2\s\s\s\s\s\s\s\s{]} & {[}\s\s\s\s\s\s\4\s\5\s\6{]} & {[}\s\s\2\s\3\s\s\s\5\s\s{]}\\
        {[}\1\s\s\s\3\s\4\s\5\s\6{]} & {[}\s\s\s\s\s\s\4\s\s\s\s{]} & {[}\s\s\2\s\s\s\4\s\5\s\s{]} & {[}\s\s\2\s\3\s\s\s\5\s\6{]} & {[}\s\s\s\s\3\s\4\s\5\s\6{]}\\
        {[}\1\s\s\s\3\s\4\s\s\s\s{]} & {[}\1\s\2\s\3\s\s\s\s\s\6{]} & {[}\s\s\2\s\3\s\s\s\5\s\s{]} & {[}\1\s\2\s\3\s\s\s\5\s\s{]} & {[}\s\s\s\s\3\s\4\s\5\s\s{]}\\
        {[}\s\s\s\s\3\s\4\s\5\s\6{]} & {[}\s\s\2\s\3\s\s\s\5\s\s{]} & {[}\s\s\2\s\s\s\s\s\5\s\s{]} & {[}\s\s\s\s\3\s\s\s\5\s\6{]} & {[}\s\s\s\s\3\s\4\s\s\s\6{]}\\
        {[}\1\s\s\s\3\s\4\s\s\s\6{]} & {[}\1\s\2\s\s\s\s\s\5\s\6{]} & {[}\s\s\2\s\3\s\s\s\s\s\s{]} & {[}\s\s\2\s\3\s\s\s\5\s\s{]} & {[}\s\s\s\s\3\s\4\s\s\s\s{]}\\
        {[}\s\s\s\s\3\s\4\s\s\s\6{]} & {[}\s\s\2\s\3\s\s\s\s\s\6{]} & {[}\s\s\2\s\s\s\s\s\s\s\s{]} & {[}\s\s\s\s\3\s\s\s\5\s\s{]} & {[}\s\s\2\s\3\s\s\s\s\s\s{]}\\
        {[}\s\s\s\s\3\s\4\s\s\s\s{]} & {[}\1\s\2\s\3\s\s\s\5\s\s{]} & {[}\1\s\s\s\3\s\4\s\5\s\6{]} & {[}\s\s\s\s\s\s\4\s\5\s\s{]} & {[}\s\s\s\s\3\s\s\s\s\s\6{]}\\
        {[}\1\s\s\s\s\s\4\s\5\s\6{]} & {[}\1\s\2\s\s\s\s\s\5\s\s{]} & {[}\1\s\s\s\3\s\4\s\5\s\s{]} & {[}\s\s\s\s\s\s\s\s\5\s\6{]} & {[}\s\s\s\s\3\s\s\s\5\s\s{]}\\
        {[}\1\s\s\s\s\s\4\s\s\s\6{]} & {[}\1\s\2\s\s\s\s\s\s\s\6{]} & {[}\s\s\s\s\3\s\4\s\5\s\s{]} & {[}\s\s\s\s\s\s\s\s\5\s\s{]} & {[}\s\s\s\s\3\s\s\s\s\s\s{]}\\
        {[}\1\s\s\s\s\s\4\s\s\s\s{]} & {[}\1\s\s\s\s\s\s\s\5\s\6{]} & {[}\1\s\s\s\3\s\s\s\5\s\s{]} & {[}\1\s\2\s\3\s\4\s\s\s\6{]} & {[}\1\s\2\s\s\s\4\s\5\s\6{]}\\
        {[}\1\s\s\s\3\s\s\s\5\s\6{]} & {[}\1\s\s\s\s\s\s\s\s\s\6{]} & {[}\s\s\s\s\3\s\s\s\5\s\s{]} & {[}\s\s\2\s\3\s\4\s\s\s\6{]} & {[}\s\s\2\s\s\s\4\s\5\s\6{]}\\
        {[}\1\s\s\s\3\s\s\s\s\s\6{]} & {[}\1\s\2\s\s\s\s\s\s\s\s{]} & {[}\1\s\s\s\s\s\s\s\5\s\6{]} & {[}\s\s\s\s\3\s\4\s\s\s\6{]} & {[}\s\s\2\s\s\s\s\s\s\s\6{]}\\
        {[}\1\s\s\s\3\s\s\s\s\s\s{]} & {[}\s\s\2\s\3\s\s\s\5\s\6{]} & {[}\s\s\s\s\3\s\s\s\s\s\s{]} & {[}\s\s\s\s\s\s\4\s\s\s\6{]} & {[}\1\s\2\s\s\s\4\s\5\s\s{]}\\
        {[}\s\s\2\s\3\s\4\s\5\s\6{]} & {[}\s\s\2\s\s\s\s\s\5\s\6{]} & {[}\1\s\s\s\3\s\s\s\s\s\s{]} & {[}\s\s\2\s\3\s\4\s\s\s\s{]} & {[}\s\s\2\s\s\s\4\s\5\s\s{]}\\
        {[}\s\s\s\s\3\s\s\s\5\s\6{]} & {[}\s\s\2\s\s\s\s\s\5\s\s{]} & {[}\1\s\s\s\s\s\4\s\5\s\6{]} & {[}\1\s\s\s\3\s\4\s\s\s\s{]} & {[}\1\s\2\s\s\s\s\s\5\s\s{]}\\
        {[}\s\s\s\s\3\s\s\s\s\s\6{]} & {[}\s\s\2\s\s\s\s\s\s\s\s{]} & {[}\s\s\s\s\s\s\4\s\5\s\6{]} & {[}\s\s\2\s\3\s\s\s\s\s\6{]} & {[}\1\s\s\s\s\s\s\s\5\s\s{]}\\
        {[}\s\s\s\s\3\s\s\s\s\s\s{]} & {[}\1\s\s\s\3\s\s\s\5\s\6{]} & {[}\1\s\s\s\s\s\4\s\5\s\s{]} & {[}\1\s\2\s\3\s\4\s\s\s\s{]} & {[}\s\s\2\s\s\s\s\s\5\s\6{]}\\
        {[}\1\s\s\s\s\s\s\s\5\s\6{]} & {[}\1\s\s\s\3\s\s\s\5\s\s{]} & {[}\1\s\s\s\s\s\s\s\5\s\s{]} & {[}\s\s\s\s\s\s\s\s\s\s\6{]} & {[}\s\s\2\s\s\s\s\s\5\s\s{]}\\
        {[}\1\s\s\s\s\s\s\s\s\s\6{]} & {[}\1\s\s\s\s\s\s\s\5\s\s{]} & {[}\s\s\s\s\s\s\4\s\5\s\s{]} & {[}\1\s\s\s\3\s\s\s\5\s\s{]} & {[}\s\s\s\s\s\s\4\s\5\s\s{]}\\
        {[}\1\s\s\s\s\s\s\s\s\s\s{]} & {[}\s\s\s\s\s\s\s\s\5\s\6{]} & {[}\s\s\s\s\s\s\s\s\5\s\s{]} & {[}\s\s\s\s\s\s\4\s\s\s\s{]} & {[}\s\s\s\s\s\s\s\s\5\s\s{]}\\
        {[}\s\s\s\s\s\s\4\s\5\s\6{]} & {[}\1\s\s\s\s\s\s\s\s\s\s{]} & {[}\1\s\s\s\3\s\4\s\s\s\6{]} & {[}\1\s\2\s\3\s\s\s\5\s\6{]} & {[}\s\s\2\s\s\s\4\s\s\s\6{]}\\
        {[}\s\s\s\s\s\s\4\s\s\s\6{]} & {[}\s\s\s\s\3\s\s\s\5\s\6{]} & {[}\1\s\s\s\s\s\4\s\s\s\6{]} & {[}\1\s\2\s\3\s\s\s\s\s\6{]} & {[}\s\s\2\s\s\s\s\s\s\s\s{]}\\
        {[}\s\s\s\s\s\s\s\s\5\s\6{]} & {[}\s\s\s\s\3\s\s\s\5\s\s{]} & {[}\1\s\s\s\s\s\s\s\s\s\6{]} & {[}\1\s\2\s\3\s\s\s\s\s\s{]} & {[}\1\s\s\s\s\s\4\s\5\s\s{]}\\
        {[}\s\s\s\s\s\s\s\s\s\s\6{]} & {[}\s\s\s\s\s\s\s\s\5\s\s{]} & {[}\s\s\s\s\s\s\4\s\s\s\6{]} & {[}\1\s\s\s\3\s\s\s\s\s\s{]} & {[}\1\s\2\s\s\s\4\s\s\s\6{]}\\
        {[}\1\s\2\s\3\s\4\s\5\s\s{]} & {[}\s\s\2\s\s\s\s\s\s\s\6{]} & {[}\1\s\s\s\3\s\4\s\s\s\s{]} & {[}\s\s\2\s\3\s\s\s\s\s\s{]} & {[}\1\s\2\s\s\s\4\s\s\s\s{]}\\
        {[}\s\s\2\s\3\s\4\s\5\s\s{]} & {[}\1\s\s\s\3\s\s\s\s\s\6{]} & {[}\s\s\s\s\3\s\4\s\s\s\s{]} & {[}\s\s\s\s\3\s\s\s\s\s\s{]} & {[}\s\s\2\s\s\s\4\s\s\s\s{]}\\
        {[}\1\s\s\s\3\s\4\s\5\s\s{]} & {[}\s\s\s\s\3\s\s\s\s\s\6{]} & {[}\1\s\s\s\s\s\4\s\s\s\s{]} & {[}\1\s\2\s\s\s\s\s\s\s\6{]} & {[}\1\s\2\s\s\s\s\s\5\s\6{]}\\
        {[}\s\s\s\s\s\s\4\s\s\s\s{]} & {[}\s\s\s\s\s\s\s\s\s\s\6{]} & {[}\1\s\s\s\s\s\s\s\s\s\s{]} & {[}\1\s\2\s\s\s\s\s\s\s\s{]} & {[}\s\s\s\s\s\s\4\s\s\s\s{]}\\
        \rule[-2.2ex]{0pt}{0pt} {[}\s\s\s\s\s\s\s\s\s\s\s{]} \rule{0pt}{0ex} & {[}\s\s\s\s\s\s\s\s\s\s\s{]} & {[}\s\s\s\s\s\s\s\s\s\s\s{]} & {[}\s\s\s\s\s\s\s\s\s\s\s{]} & {[}\s\s\s\s\s\s\s\s\s\s\s{]} \\
        \hline
    \end{tabular}
    
    \vspace{5mm}
    
    \caption{Linear orders $L_1, L_2, L_3, L_4, L_5$ on all subsets of $[6]$ forming the Boolean realizer of $\lat{6}$. Each column corresponds to one linear order. The greatest element in an order is the top one.}\label{tab:boolean_realizer}
    \label{tab:my_label}
\end{table}

%% file: s.multisets.tex






In this section, we prove \cref{th:mlat_exact_bdim}, that is, for every positive integer $n$, $\bdim(\mlat{n}) = \dim(\mlat{n}) = n$.
We have $\bdim(\mlat{n}) \leq \dim(\mlat{n})$, and by \cref{prop:dim_of_lattices}, $\dim(\mlat{n}) = n$.
Therefore, in order to prove \cref{th:mlat_exact_bdim}, it suffices to show that $n \leq \bdim(\mlat{n})$.
To this end, we analyze a certain class of subposets of $\mlat{n}$. 

For all positive integers $n$ and $m$, we define $\mn{n}{m}$ to be the subposet of $\mlat{n}$ induced by all multisets such that every element of a multiset has multiplicity less than $m$. The number of elements in $\mn{n}{m}$ is equal to $m^n$.
Moreover, $\mn{n}{2}$ is isomorphic to~$\lat{n}$.

\begin{lemma}\label{lemma:mn_lower_bound}
    For all positive integers $n, m$, we have $\bdim(\mn{n}{m}) \geq \frac{n \log m}{\log(nm - n + 1)}$.
\end{lemma}

\begin{proof}
    Let $n, m$ be positive integers and assume that $\bdim(\mn{n}{m}) = d$ for some positive integer~$d$. It follows that there exist linear orders $L_1, \dots, L_d$ and $\phi: \{0, 1 \}^{d} \rightarrow \{0, 1\}$ such that $((L_1, \dots, L_d), \phi)$ is a Boolean realizer of $\mn{n}{m}$.
    
    We define $\cgS$ to be the set of all multisets in $\mn{n}{m}$ consisting of exactly one element with a positive multiplicity. 
    For every multiset $A$ in $\mn{n}{m}$ and for every $i \in [d]$, we define $s_i(A)$ to be the number of elements in $\cgS$ that are less than $A$ in $L_i$. 
    Clearly, $0 \leq s_i(A) \leq |\cgS|$.
    Note that $|\cgS| = n(m-1)$.
    For every element $A$ in $\mn{n}{m}$, let $s(A) =(s_1(A), \dots, s_d(A))$ be its \emph{signature}. 
    We claim that all elements of $\mn{n}{m}$ have distinct signatures. 
    See \cref{fig:2} for an illustration of the following argument.
    
    Suppose, contrary to our claim that there exist distinct $A, B$ in $\mn{n}{m}$ with the same signatures, that is, $s(A) = s(B)$. 
    Fix some $i \in [d]$. 
    Since $s_i(A) = s_i(B)$ and $L_i$ is a fixed linear order, both $A$ and $B$ are greater than the exact same set of elements from $\cgS$ in $L_i$.
    In particular, for every $S \in \cgS$, 
    \begin{align*}\label{eq:sametuple}
        [S \leq_{L_i} A]_{i=1}^d = [S \leq_{L_i} B]_{i=i}^d.
    \end{align*}
    Since $A$ and $B$ are distinct, there is $x \in [n]$ that occurs in $A$ and $B$ with different multiplicities. 
    Without loss of generality, assume that the multiplicity of $x$ in $A$ is greater than the multiplicity of $x$ in $B$.
    Let $T$ be the multiset consisting of exactly $x$ with the multiplicity equal to the multiplicity of $x$ in $A$.
    Note that $T \in \cgS$, and therefore, $ [T \leq_{L_i} A]_{i=1}^d = [T \leq_{L_i} B]_{i=i}^d$, which yields, $T \leq A$ in $\mn{n}{m}$ if and only if $T \leq B$ in $\mn{n}{m}$, a contradiction.
    We conclude that elements in $\mn{n}{m}$ have distinct signatures. 
    Therefore,
        \[(|\cgS| + 1)^d \geq |\mn{n}{m}|.\]
    Since $|\cgS| = n(m-1)$ and $|\mn{n}{m}| = m^n$, we have $(n(m-1)+1)^d \geq m^n$, and finally, $d \geq \frac{n \log m}{\log(nm - n + 1)}$.
\end{proof}

    \begin{figure}[h]
        \centering
        \includegraphics[scale=1.2]{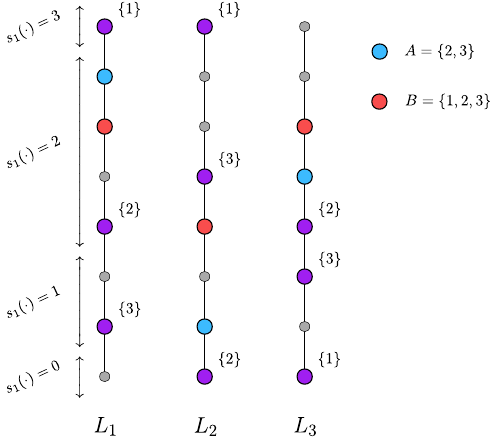}
        \caption{
        Consider the poset $\mn{3}{2} = \lat{3}$.
        $L_1, L_2, L_3$ are some linear orders on elements of the poset (in the figure the least element is on the bottom).
        The singletons $\cgS = \{\{1\},\{2\},\{3\}\}$ are highlighted with purple color.
        We fix $A = \{2,3\}$ (in blue) and $B = \{1,2,3\}$ (in red).
        We have $s_1(A) = 2$, since $A$ is greater than $\{2\}$ and $\{3\}$ and less than $\{1\}$.
        Similarly, one can check that $s_2(A) = 1$ and $s_3(A) = 3$.
        It follows that the signature of $A$ is equal to $(2,1,3)$.
        The signature of $B$ turns out to be the same.
        We claim that this prevents $L_1,L_2,L_3$ from being a Boolean realizer of $\lat{3}$ regardless of $\phi : \{0,1\}^{3} \rightarrow \{0,1\}$.
        Indeed, $1 \in B \backslash A$, however, the relations between $\{1\}$ and $A$ are the same as the relations between $\{1\}$ and $B$ in $L_1,L_2,L_3$. This is a contradiction since $\{1\} \leq B$ and $\{1\} \not\leq A$.
        }
        \label{fig:2}
    \end{figure}

Let us remark on a possible generalization of \cref{lemma:mn_lower_bound}.
A subset $D$ of elements of a poset $P$ is \emph{distinguishing for $P$} if for every two distinct elements $x,y$ in $P$ there is $z \in D$ such that the relation between $z$ and $x$ is different from the relation between $z$ and $y$.
The set $\cgS$ in the proof above is distinguishing for the poset $\mlat{n,m}$ and it is actually the only property that we use.
Therefore, \Cref{lemma:mn_lower_bound} can be easily generalized to the following statement: if a poset $P$ contains a distinguishing set $D$, then $\bdim(P) \geq \frac{\log(|P|)}{\log(|D|+1)}$.



\cref{lemma:mn_lower_bound} applied with $m = 2$ gives a lower bound on $\bdim(\lat{n})$.

\begin{corollary}\label{cor:lat_bdim_lower_bound}
    For every positive integer $n$, we have $\bdim(\lat{n}) \geq \frac{n}{\log(n + 1)}$.
\end{corollary}

For every positive integer $n$, the limit $\lim_{m \to \infty} \frac{n \log m}{\log(nm - n + 1)}$ is equal to $n$. It follows that $\bdim(\mn{n}{m}) = n$ for a large enough $m$, and so, $\bdim(\mlat{n}) = n$, which concludes the proof of \cref{th:mlat_exact_bdim}. 
To be more precise, one can compute how large $m$ should be.

\begin{proposition}\label{prop:bdim_mn_equal_to_n}
    For every integer $n$ with $n \geq 2$, we have $\bdim(\mn{n}{n^{n-1}}) = n$.
\end{proposition}


    



%% file: s.open.tex
To sum up, we list a few related open problems. 
By \cref{cor:lat_bdim_lower_bound} and by \cref{th:technical_upper_bound}, for every positive integer $n$ large enough, we have $\frac{n}{\log(n+1)} \leq \bdim(\lat{n}) \leq \left\lceil\frac{5}{6} n\right\rceil$.
This is not tight, and therefore, we state the following question.

\begin{myquestion}\label{q:bdim_lat}
    What is the order of magnitude of $\bdim(\lat{n})$?
\end{myquestion}

A proper examination of the proof of \cref{lemma:mn_lower_bound} in the case of $m = 2$ (recall that $\mn{n}{2}$ is isomorphic to $\lat{n}$) shows that we never use the fact that a fixed Boolean realizer detects comparabilities between $A,B \subset [n]$ with $|A|,|B| > 1$. 
In fact, to show the lower bound, we utilize only the detection of comparabilities between singletons and other subsets of $[n]$.
Moreover, the singletons play a special role since they form the smallest distinguishing set in $\lat{n}$.
This motivates us to consider a relaxed version of the Boolean dimension problem for Boolean lattices. Namely, let $\bdim(1,\lat{n})$ be the minimum positive integer $d$ such that there exists a sequence of linear orders $( L_1, \dots, L_d )$ of elements of $\lat{n}$ and $\phi: \{0, 1\}^{d} \rightarrow \{ 0, 1 \}$ such that for every $x \in [n]$ and $A \in \lat{n}$, we have $\{x\} \leq A$ in $\lat{n}$ if and only if $\phi \left([\{x\} \leq_{L_i} A]_{i = 1}^d \right) = 1$.
Note that such relaxation was considered for dimension -- see e.g.\ \cite{BKKT94, K96, K99}.
Clearly, $\bdim(1,\lat{n}) \leq \bdim(\lat{n})$ and due to the discussion above, $\frac{n}{\log(n+1)} \leq \bdim(1,\lat{n})$.
However, even though bounding $\bdim(1,\lat{n})$ is simpler than bounding $\bdim(\lat{n})$, there is no known better upper bound.

\begin{myquestion}\label{q:bdim_1_n}
    What is the order of magnitude of $\bdim(1,\lat{n})$? Is it the same as the order of magnitude of $\bdim(\lat{n})$?
\end{myquestion}

 The last question is related to the finite subposets of $\mlat{n}$. This question was already stated in \mbox{\cite[Section~3.4]{BPSTT19}}.

\begin{myquestion}\label{q:mn}
    For a positive integer $n$, let $f(n)$ be the least positive integer $m$ with $\bdim(\mn{n}{m}) = n$.
    What is the order of magnitude of $f(n)$?
\end{myquestion}

 By \cref{prop:bdim_mn_equal_to_n} and \cref{th:upper_bound}, for an integer $n$ large enough, we have $3 \leq f(n) \leq n^{n-1}$.
 This leaves a substantial gap for further research.

%% file: s.code.tex
We provide a short Python code that verifies whether the five linear orders in \cref{tab:boolean_realizer} and the function $\phi$ defined in the proof of \cref{lem:bdim_of_B_6} form a Boolean realizer of $\lat{6}$.
Beneath the code, we give the formatted linear orders so that they can be directly inputted into the script for verification.

\vspace{2mm}

\scriptsize

\begin{tcolorbox}[
    colback=white, 
    colframe=gray,
    sharpish corners
]
\begin{lstlisting}[language=Python]
import numpy as np

def phi(orders, A, B):
    t = np.array([], dtype=int)
    for order in orders:
        t = np.append(t, [1 if order.index(A) <= order.index(B) else 0])
    return (t == 0).sum() <= 1;

def is_subset(A, B):
    return (A | B) == B

def verify(orders):
    for A in range(1<<6):
        for B in range(1<<6):
            if (phi(orders, A, B) != is_subset(A, B)):
                raise Exception("Provided set of orders is not a Boolean realizer of B_6.")

orders = []
for i in range(5):
  orders.append(list(map(int, input().split())))

verify(orders)
\end{lstlisting}

\tcblower

Attached file: \textattachfile[color=blue]{verify.py}{verify.py}

\end{tcolorbox}

\vspace{2mm}

\normalsize 

The input consists of $5$ lines. Each represents one of the linear orders in \cref{tab:boolean_realizer}. The subsets of $[6]$ are converted to corresponding decimal integers based on their binary representations. For example, the number $13$ corresponds to the set $\{ 1, 3, 4 \}$, since the binary representation of $13$ on $6$ bits is $001101_2$.

\scriptsize

\vspace{4mm}

\begin{tcolorbox}[
    colback=white, 
    colframe=gray,
    sharpish corners
]
\begin{lstlisting}[language=Python]
0 8 29 30 31 32 48 40 56 1 33 49 4 36 52 62 5 37 53 9 41 57 12 44 45 60 13 61 2 34 10 42 6 14 58 38 46 3 35 59 7 39 11 15 43 47 16 18 28 19 26 27 22 50 54 17 51 20 21 23 24 25 55 63

0 32 36 37 34 16 20 52 1 48 17 21 53 2 18 50 54 3 33 49 35 19 23 38 51 22 39 8 55 9 24 25 10 26 11 27 28 40 42 41 43 56 57 58 59 4 7 12 44 60 6 14 46 5 13 45 47 61 15 30 62 29 31 63

0 1 9 12 13 40 33 41 45 16 24 17 25 56 57 5 4 49 20 21 28 29 61 2 6 18 22 26 3 7 19 10 11 27 23 14 15 31 32 37 53 30 34 35 42 36 46 43 48 47 50 58 51 59 38 39 52 54 55 8 44 60 62 63

0 3 35 4 6 5 7 39 55 8 21 32 15 38 13 14 40 44 46 47 16 48 24 20 22 52 23 54 56 12 28 60 2 34 18 50 10 42 26 30 58 9 62 1 11 17 25 27 29 19 31 33 41 43 49 51 57 59 36 37 53 45 61 63

0 8 51 10 11 43 25 2 42 16 24 18 50 17 19 26 27 34 58 59 4 20 36 6 12 44 28 60 22 38 54 14 46 62 1 30 5 7 23 9 13 15 32 33 35 37 41 45 3 39 47 48 49 57 52 21 53 29 31 40 56 55 61 63 
\end{lstlisting}

\tcblower

Attached file: \textattachfile[color=blue]{orders.in}{orders.in}

\end{tcolorbox}